\documentclass{article}

\usepackage{microtype}
\usepackage{graphicx}
\usepackage{subfigure}
\usepackage{booktabs} 
\usepackage{amssymb,amsmath,amsthm}
\usepackage{enumitem}

\usepackage{hyperref}

\usepackage[accepted]{icml2020}

\DeclareMathOperator\argmax{argmax}
\newtheorem{theorem}{Theorem}[section]
\newtheorem{lemma}[theorem]{Lemma}
\newtheorem{proposition}[theorem]{Proposition}

\newtheorem{op}[theorem]{Open Problem}

\numberwithin{equation}{section}

\icmltitlerunning{Buckley-Osthus \& block preferential attachment models}

\begin{document}

\twocolumn[
\icmltitle{The Buckley-Osthus model and the block preferential attachment model: statistical analysis and application}



\icmlsetsymbol{equal}{*}

\begin{icmlauthorlist}
\icmlauthor{Xin Guo}{UCB}
\icmlauthor{Fengmin Tang}{Stanford}
\icmlauthor{Wenpin Tang}{Columbia}
\end{icmlauthorlist}

\icmlaffiliation{Columbia}{Department of Industrial Engineering and Operations Research, Columbia University, USA}
\icmlaffiliation{UCB}{Department of Industrial Engineering and Operations Research, University of California, Berkeley, USA}
\icmlaffiliation{Stanford}{ICME, Stanford University, USA}

\icmlcorrespondingauthor{Wenpin Tang}{wt2319@columbia.edu}

\icmlkeywords{Buckley-Osthus model, consistency, block preferential attachment model, maximum likelihood estimate}

\vskip 0.3in
]



\printAffiliationsAndNotice{\icmlEqualContribution} 

\begin{abstract}
This paper is concerned with statistical estimation of two preferential attachment models: 
the Buckley-Osthus model and a new block preferential attachment model. 
We prove that the maximum likelihood estimates for both models are consistent. 
We perform simulation studies to corroborate our theoretical findings. 
We also apply both models to study the evolution of a real-world network. 
A list of open problems are presented.
\end{abstract}

\section{Introduction}
Networks are ubiquitous.
A network consists of elements or actors represented by nodes or vertices, with  interactions modeled by links or edges.
Network studies are usually data-driven via graph theoretical analysis, aiming to determine the influence of its constituents.
In the era of data deluge, there is a surge of interest in exploring the features of large networks such as the Internet \cite{GGP, BHMBM}, social media \cite{Cen10, WJLC}, biological systems \cite{JMBO, BGL11}, and more recently blockchains \cite{N08, C16}.

Large networks are complex, and their intricate structures can be modeled by random graphs.
The best-known random graph model is the {\em Erd\"{o}s-R\'enyi graph} \cite{ER59, ER61}, where any two nodes are linked independently with a fixed probability. 
See \cite{B01, Durrett07} for further results of the Erd\"{o}s-R\'enyi graph such as the limiting degree distribution and emergence of the giant component.

Despite its simple form and wide applications, the Erd\"{o}s-R\'enyi model is often criticized  for the following two reasons.
\begin{itemize}
\item The Erd\"{o}s-R\'enyi model does not take  into account the heterogeneity, whereas
a real-world network is heterogeneous, composed of a few communities or groups and exhibiting different characteristics across these communities.

\item The Erd\"{o}s-R\'enyi graph has an unrealistic Poisson degree distribution, whereas
in many real-world networks, the degree of a typical node is observed to be power-law or heavy-tailed distributed.
\end{itemize}
To tackle the first point, \cite{HLL} introduced the {\em stochastic blockmodel} (SBM) 
which is a generalization of the Erd\"{o}s-R\'enyi graph with community structure.
In the past decade, there has been considerable progress on the SBM including community detection \cite{NG02, BC09, MNS15, ABH16, MNS18} and statistical estimation \cite{DPR08, BCL11, BCCZ13, WB17}.
To address the second point, \cite{BA99} proposed the {\em preferential attachment model} (PAM), which is an instance of the {\em scale-free network}.
It attaches a new node to existing ones according to the popularity, i.e., the degrees of the existing nodes, and 
the resulting degree distribution follows a power law.
See also \cite{KN11, ZLZ12, CLX18} for the degree-corrected SBM.

Since the work of Barab\'{a}si and Albert, preferential attachment and related models have attracted much attention from combinatorics and probability communities, see e.g. \cite{BRST01, BBCR, BR03, BR04, PRR13, BSZ15, BMR}.
But it was not until recently that statistical estimation of the PAM was investigated.
\cite{GV17, GVCV, WR19} considered estimates of an undirected PAM, while the counterpart of a directed PAM was studied by \cite{WWDR}.

This paper is concerned with two classes of PAM: 
\begin{itemize}
\item
the {\em Buckley-Osthus model} \cite{BO04}, which is a one-parameter PAM allowing for self-loops;
\item
a new model that adds a community structure to the PAM, which we call  {\em block preferential attachment model} (BPAM).
\end{itemize}
These two models capture different features of the underlying network.
The parameter in the Buckley-Osthus model identifies the exponent of the power-law degree distribution, and the community structure in the BPAM characterizes the interactions across various groups.
The relation between the BPAM and the PAM is the same as that between the SBM and the Erd\"{o}s-R\'enyi graph.

Meanwhile, inn contrast to the SBM, the BPAM hinges on the order of nodes added, thus describing the evolution of the network.
Thus, the BPAM naturally combines the idea of the SBM and the PAM, and may overcome the two drawbacks of the Erd\"{o}s-R\'enyi model.
Though some related characteristics of the BPAM  are implicitly discussed in the work of \cite{J13, HS18}, it is the first time in this work the BPAM is explicitly formulated.
We also provide statistical analysis under this framework.

In this paper, we show the asymptotic consistency of the maximum likelihood estimate (MLE)
for the aforementioned two models.
Precisely, we prove that (1) the MLE of the Buckley-Osthus model is asymptotically normal, 
(2) the MLE of the BPAM with observed community memberships is consistent.
Contrary to the SBM, the likelihood of either the Buckley-Osthus model or the BPAM does not belong to the exponential family.
Thus, the proof of these results does not follow from standard theory, and requires extra mathematical tools and careful analysis.
We also compare the BPAM with the SBM, and apply both models to study the evolution of real-world networks, and in particular the Bitcoin network.
We present a list of open problems such as community detection and efficient estimation of the BPAM.
We hope that this work serves as a first step towards understanding the BPAM, a natural model for modeling network evolution with community structures.

The rest of this paper is organized as follows.
\begin{itemize}
\item
In Section \ref{s2}, we present the models of interest, and state the main results.
\item
In Section \ref{s3}, we conduct simulation studies to corroborate our theoretical findings.
We also apply the BPAM to a real-world data -- Bitcoin network.
\item
In Section \ref{s4}, we conclude with further extensions of the PAM, and a comparison of the BPAM and the SBM.
There a list of open problems are presented.
\end{itemize}
The proofs of the lemma are given in the Supplementary Material.

\section{Models and main results}
\label{s2}
\subsection{Preferential attachment models}
\label{s21}

The PAM was proposed by Barab\'{a}si and Albert \cite{BA99}, capturing the idea that popular nodes get more attracted than less popular ones. 
It is closely related to the Yule model \cite{Yule25}, the Matthew effect \cite{Merton68}, the Price model \cite{Price} and the Chinese restaurant process \cite{Aldous85}.
Roughly speaking, it generates the graph in a sequential way by attaching a new node to existing ones with probability proportional to the degree of those nodes. 
Thus, the resulting degree distribution has a power law.
As pointed out in \cite{BRST01}, the original definition of the PAM is mathematically ambiguous.
A more precise description is given as follows.
\begin{itemize}
\item
Start with two nodes linked by $m \ge 1$ edges.
\item
When adding a new node, add edges one a time, with the second and subsequent edges performing preferential attachment using the updated degrees.
\end{itemize}
This construction allows to reduce the PAM for any $m \ge 1$ to $m = 1$ by collapsing nodes $(k - 1)m + 1, \ldots, km$ to node $k$.
See also \cite{BB01, M02, BBCR, BR03} for various extensions of the PAM. 
In the sequel, we focus on the undirected linear PAM.

An important variant of the PAM was introduced in \cite{BR04}, which is referred to as the {\em Linearized Chord Diagram} (LCD) or the {\em Bollob\'{a}s-Riordan model}.
It starts with a single node, labelled $1$, with a self-loop.
For $i \ge 1$, node $i+1$ is attached to the graph by the following rule:
\begin{equation}
\label{eq:LCD}
\def\arraystretch{1.5}
\mathbb{P}(i+1 \sim v)  = \left\{ \begin{array}{rcl}
\frac{d_i(v)}{2i+1} & \mbox{for}
& v \le i, \\ 
\frac{1}{2i+1} & \mbox{for} & v = i+1,
\end{array}\right.
\end{equation}
where $d_i(v)$ is the degree of the node $v$ at time $i$.
It can be shown that the graph under the dynamics \eqref{eq:LCD} has a degree distribution $P(k)$ proportional to $k^{-3}$.
A more general model was proposed \cite{BO04}, allowing the degree power as a parameter.
Precisely, the attachment rule is given by
\begin{equation}
\label{eq:BO}
\def\arraystretch{1.5}
\mathbb{P}(i+1 \sim v)  = \left\{ \begin{array}{rcl}
 \frac{d_i(v) + a - 1}{(a+1)i + a} & \mbox{for}
& v \le i, \\ 
\frac{a}{(a+1)i+a} & \mbox{for} & v = i+1.
\end{array}\right.
\end{equation}
By taking $a = 1$, we get the LCD model. 
The degree distribution $P(k)$ of the Buckley-Osthus model is proportional to $k^{-2 -a}$.
A nice property of the Buckley-Osthus model is the exchangeability, i.e. the graph distribution does not depend on the order of nodes added. 
Given a graph of $n$ nodes and edges, the Buckley-Osthus likelihood is 
\begin{equation}
\label{eq:likelihoodBO}
L^{BO}_n(a): = \frac{\prod_v a^{(d(v) - 1)}}{a(2a+1)(3a+2) \cdots (na+n-1)},
\end{equation}
where $a^{(n)}: = \prod_{k = 0}^{n-1} (a+k)$ is the {\em Pochhammer rising factorial}, with the convention $a^{(0)}: = 1$,
and $d(v)$ is the total degree of the node $v$.
One can fit the Buckley-Osthus model by MLE, and a natural question is whether the MLE is consistent.
\subsection{Block preferential attachment model}
\label{s22}

In view of the stochastic blockmodel, it is natural to add a community structure to the PAM. 
It allows for extra parameters to model the incentive of attachment between different communities.
Precisely, the attachment rule is as follows.
\begin{itemize}[itemsep = 3 pt]
\item
There are $K$ communities $C_1, \ldots, C_K$.
Each added node belongs to community $C_j$ with probability $\pi_j$, $1 \le j \le K$ with $\sum_{j = 1}^K \pi_j = 1$.
The memberships of nodes are independent.
\item
$(\gamma_{kl}; \, 1 \le k, l \le K)$ is a symmetric matrix representing the interaction between communities. 
The probability that node $i+1$ is attached to $v$ is proportional to $\gamma_{kl} d_i(v)$ with $i+1 \in C_k$, $v \in C_l$.
\end{itemize}
In this case,
 \begin{equation}
 \label{eq:BPA}
 \def\arraystretch{1.5}
\mathbb{P}(i+1 \sim v)  = \left\{ \begin{array}{rcl}
\frac{\gamma_{kl} d_i(v)}{\sum_{v'} \gamma_{kl'}d_i(v') + \gamma_{kk}} & \mbox{for}
& v \le i, \\ 
\frac{\gamma_{kk}}{\sum_{v'} \gamma_{kl'}d_i(v') + \gamma_{kk}} & \mbox{for} & v = i+1.
\end{array}\right.
\end{equation}
By taking $\gamma_{kl}$'s all equal, we get the LCD model.
It is easy to see that the attachment probability \eqref{eq:BPA} is homogenous in $(\gamma_{kl}; \, 1 \le k, l \le K)$.
One can take for instance $\gamma_{11} = 1$ for normalization.

The main difference between the PAM and the BPAM is that the latter is not exchangeable.
That is, the graph distribution depends on the order of nodes added. 
To see this, we need some notations.
For $n \ge 1$ and $1 \le i, j \le K$, let
\begin{itemize}[itemsep = 3 pt]
\item
$T^n_j$ be the number of nodes belonging to community $C_j$ up to time $n$;
\item
$N^n_j$ be the sum of degrees of nodes belonging to $C_j$ up to time $n$;
\item
$M_{ij}^n$ be the number of edges linking a node in $C_i$ with one in $C_j$.
\end{itemize}
For $1 \le k \le n$, let $l_k$ be the membership of node $k$, i.e. $l_k = j$ if node $k$ belongs to community $j$.
The BPAM likelihood is given by
\begin{equation}
\label{eq:BPAlikelihood}
L^{BPA}_n({\pmb \pi}, {\pmb \gamma}): = 
\prod_{j = 1}^K \pi_j^{T_j^n}  \frac{\prod_{i, j = 1}^K \gamma_{ij}^{M_{ij}^n} \prod_{v} (d(v) -1) !}{\prod_{k = 1}^n \left(\sum_{j = 1}^K \gamma_{l_k j} N_j^{k-1} + \gamma_{l_k l_k} \right)}.
\end{equation}
It can be seen from the likelihood \eqref{eq:BPAlikelihood} that the denominator is node order dependent, while the numerator is not.
Thus, the graph likelihood depends on the whole history of the network expansion, not merely the final configuration.
A first question is whether the MLE of $L_n^{BPA}$ is consistent.

By letting $t_k$ be the membership of the node attached by node $k$, the BPAM likelihood \eqref{eq:BPAlikelihood} can be rewritten as
\begin{equation}
\label{eq:BPAlikelihood2}
L^{BPA}_n({\pmb l}; \, {\pmb \pi}, {\pmb \gamma}) = 
\prod_{k = 1}^n \pi_{l_k}  \frac{\prod_{k=1}^n \gamma_{l_k t_k} \prod_{v} (d(v) -1) !}{\prod_{k = 1}^n \left(\sum_{j = 1}^K \gamma_{l_k j} N_j^{k-1} + \gamma_{l_k l_k} \right)}.
\end{equation}
Often the memberships of nodes are not observed, and this leads to considering the marginal likelihood:
\begin{equation}
\label{eq:marginalPA}
G_n^{BPA}({\pmb \pi}, {\pmb \gamma}): = \sum_{{\pmb l} \in \{1, \dots, K\}^n} L_n^{BPA}({\pmb l}; \, {\pmb \pi}, {\pmb \gamma}).
\end{equation}
Note that the marginal likelihood \eqref{eq:marginalPA} is invariant under relabeling the names of communities, i.e. $G_n^{BPA}(\Pi{\pmb \pi}, \Pi {\pmb \gamma} \Pi^T) = G_n^{BPA}({\pmb \pi}, {\pmb \gamma})$ for $\Pi$ a permutation matrix. 
It is also interesting to ask whether the MLE of $G_n^{BPA}$ is consistent.
\subsection{Theoretical results}
\label{s23}

Now we present the main results -- consistency of the MLE for both the Buckley-Osthus model and the BPAM.
Recall the Buckley-Osthus likelihood from \eqref{eq:likelihoodBO}. 
We restrict to the domain $\mathcal{D}: = [\varepsilon, M] \subset (0, \infty)$, and consider the following scaled log-likelihood:
\begin{multline}
\label{eq:loglikeOB}
\ell^{BO}_n(a):= \frac{1}{n} \Bigg( \sum_{v \, : \, d(v) \ge 2} \sum_{k = 0}^{d(v)-2} \log(a+k)  \\
-  \sum_{k = 1}^n \log\left(a + \frac{k-1}{k}\right) \Bigg).
\end{multline}
The MLE of the Buckley-Osthus model is defined by $\widehat{a}^{BO}_n : = \argmax_{a \in \mathcal{D}} \ell^{BO}_n(a)$.
Denote $a_0$ to be the true value of $a$.
Our first result shows the consistency of $\widehat{a}^{BO}_n$.
\begin{theorem}
\label{thm:consistency}
Assume that $a_0 \in \mathcal{D}$.
Then $\widehat{a}^{BO}_n \rightarrow a_0$ almost surely as $n \rightarrow \infty$
\end{theorem}

Moreover, we can prove the asymptotic normality of $\widehat{a}^{BO}_n$.
\begin{theorem}
\label{thm:asympnormal}
As $n \rightarrow \infty$,
\begin{equation}
\sqrt{n}(\widehat{a}^{BO}_n - a_0) \stackrel{(d)}{\longrightarrow} \mathcal{N}\left(0, \frac{\sigma^2}{\beta^2}\right),
\end{equation}
where 
\begin{equation}
\label{eq:sigma}
\sigma^2 : = \sum_{k \ge 0} \frac{p_{>k+1}}{(a_0+k)^2} - \frac{2}{a_0+1} \sum_{k \ge 0} \frac{p_{>k+1}}{a_0+k} + \frac{1}{(a_0+1)^2},
\end{equation}
and 
\begin{equation}
\label{eq:beta}
\beta: = \sum_{k \ge 0} \frac{p_{>k+1}}{(a_0+k)^2} - \frac{1}{(a_0+1)^2},
\end{equation}
with 
\begin{equation}
\label{eq:pkdef}
p_k : = \frac{(a_0+1) a_0^{(k-1)}}{(2a_0+1)^{(k)}},
\end{equation}
and $p_{>k} : = \sum_{j = k+1}^{\infty} p_j$.
\end{theorem}

Now let us move onto the BPAM whose likelihood is given by \eqref{eq:BPAlikelihood}.
We consider the following scaled log-likelihood:
\begin{multline}
\label{eq:BPAlog}
\ell^{BPA}_n({\pmb \pi}, {\pmb \gamma}): = \frac{1}{n} \Bigg( \sum_{j = 1}^K T_j^n \log \pi_j  + \sum_{i, j = 1}^K M^n_{ij} \log \gamma_{ij} \\
- \sum_{k = 1}^n \log \left(\sum_{j = 1}^K \gamma_{l_k j} \frac{N^{k-1}_j}{k} \right) \Bigg).
\end{multline}
We write $\ell^{BPA}_n({\pmb \pi}, {\pmb \gamma}) = \ell^{BPA}_n({\pmb \pi}) + \ell^{BPA}_n({\pmb \gamma})$ where
\begin{equation}
\label{eq:BPApi}
\ell^{BPA}_n({\pmb \pi}): = \frac{1}{n} \sum_{j = 1}^K T_j^n \log \pi_j 
\end{equation}
and
\begin{multline}
\label{eq:BPAgamma}
\ell^{BPA}_n({\pmb \gamma}): = 
\frac{1}{n} \Bigg(  \sum_{i, j = 1}^K M^n_{ij} \log \gamma_{ij} \\
- \sum_{k = 1}^n \log \left(\sum_{j = 1}^K \gamma_{l_k j} \frac{N^{k-1}_j}{k} \right) \Bigg).
\end{multline}
The MLE $(\widehat{\pmb \pi}, \widehat{\pmb \gamma})$ is defined by $\widehat{\pmb \pi}: = \argmax_{{\pmb \pi} \in \mathcal{D}}\ell^{BPA}_n({\pmb \pi})$ over the set of probability distributions $\mathcal{D}$, and 
$\widehat{\pmb \gamma}: = \argmax_{{\pmb \gamma} \in \mathcal{S}} \ell^{BPA}_n({\pmb \gamma})$ over the set of symmetric stochastic matrices $\mathcal{S}$.
Denote $({\pmb \pi}^0, {\pmb \gamma}^0)$ to the true parameter values.
The main result is the consistency of $(\widehat{\pmb \pi}, \widehat{\pmb \gamma})$.
\begin{theorem}
\label{thm:consistencyBPA}
We have $\widehat{\pmb \pi} \rightarrow {\pmb \pi}^0$, $\widehat{\pmb \gamma} \rightarrow {\pmb \gamma}^0$ almost surely as $n \rightarrow \infty$.
\end{theorem}
\subsection{Roadmap to the proofs}
\label{s24}

Let us start with the consistency of $\widehat{a}^{BO}_n$. 
For $k \ge 1$, let $Z_k^n$ be the number of nodes of degree $k$, and $Z^n_{>k} = \sum_{j = k+1}^{\infty} Z^n_j$ be the number of nodes of degree greater than $k$.
The log-likelihood \eqref{eq:loglikeOB} can be expressed as
\begin{equation}
\label{eq:loglikeOB2}
\ell^{BO}_n(a) = \sum_{k \ge 0}\frac{Z^n_{>k+1}}{n}  \log(a+k)  - \frac{1}{n} \sum_{k = 1}^n \log\left( a + \frac{k-1}{k}  \right).
 \end{equation}
The key idea to prove Theorem \ref{thm:consistency} is that $Z^n_k/n$ converges almost surely to $p_k$, which forms a probability distribution.
\begin{proposition}
\label{thm:convp}
For $k \ge 1$, $Z^n_k/n \rightarrow p_k$ almost surely, where
$p_k$ is defined by \eqref{eq:pkdef}.
\end{proposition}
\begin{proof}
For $k \ge 1$, let $N^n_k : = \mathbb{E}Z^n_k$. We prove that $N^n_k/n \rightarrow p_k$ as $n \rightarrow \infty$.
The idea is to establish the recursion for $N_k^n$. 
For $k = 1$,
\begin{equation}
\label{eq:N1}
N_1^{n+1} - N^n_1 = \frac{(a_0+1)n}{(a_0+1)n + a_0} - \frac{a_0 N_1^n}{(a_0+1)n + a_0},
\end{equation}
where the first term on the r.h.s. comes from attaching the new node to any existing one, and the second term is due to loss of a degree one node being attached by the new node.
For $k = 2$,
\begin{multline}
\label{eq:N2}
N_2^{n+1} - N_2^n = \frac{a_0}{(a_0+1)n + a_0} + \frac{a_0 N_1^n}{(a_0+1)n+a_0} \\
- \frac{(a_0+1)N_2^n}{(a_0+1)n+a_0},
\end{multline}
where the first term on the r.h.s. is due to creation of a loop by the new node, the second term due to creation of a degree two node by attaching the new node to a degree one node, and the third term due to loss of a degree two node being attached by the new node.
Similarly, for $k \ge 3$,
\begin{equation}
\label{eq:N3}
N^{n+1}_k - N^n_k = \frac{(k+a_0-2) N^{n}_{k-1}}{(a_0+1)n + a_0} - \frac{(k+a_0-1) N^n_k}{(a_0+1)n + a_0}.
\end{equation}
From \eqref{eq:N1}, we get
\begin{equation*}
N_1^{n+1} = \left( 1 - \frac{a_0}{(a_0+1)n + a_0}\right) N_1^n + \frac{(a_0+1)n}{(a_0+1)n + a_0}.
\end{equation*}
By Lemma 4.1.2 in \cite{Durrett07}, 
$N_1^n/n \rightarrow 1/(1 + \frac{a_0}{a_0+1}) = p_1$ as $n \rightarrow \infty$. 
Similarly, we have $N^k_n/n \rightarrow p_k$ for all $k$. 
During a time epoch, the number of nodes with degree $k$ differs at most two.
A standard argument by combining the Azuma-Hoeffding inequality and the Borel-Cantelli lemma yields the desired result.
\end{proof}

This leads to considering the limit of \eqref{eq:loglikeOB2}: 
\begin{equation}
\label{eq:limit}
\ell^{BO}_{\infty}(a): = \sum_{k \ge 0} p_{>k+1} \log(a+k) - \log(a+1),
\end{equation}
where $p_{> k+1} : = \sum_{j = k+2}^{\infty} p_j$. 
Next we show that $\ell^{BO}_{\infty}$ attains its unique maximum at $a_0$, from which we derive the consistency of $\widehat{a}^{BO}_n$.
\begin{lemma}
\label{prop:analysis}
The function $\ell_{\infty}^{BO}(\cdot)$ has a unique maximum at $a_0$. Moreover, for any $\varepsilon > 0$,
\begin{equation}
\label{eq:ellest}
\sup_{a > \varepsilon} |\ell_n^{BO'}(a) - \ell_{\infty}^{BO'}(a)| \rightarrow 0 \quad a.s.
\end{equation}
\end{lemma}
\begin{proof}[Proof of Theorem \ref{thm:consistency}]
Fix $\eta > 0$. According to Lemma \ref{prop:analysis}, there exists $m_{\eta} > 0$ such that
\begin{equation*}
\ell'_{\infty}(a) > m_{\eta} \mbox{ for } a \in [\varepsilon, a_0 - \eta]
\end{equation*}
and 
\begin{equation*}
\ell'_{\infty}(a) < - m_{\eta} \mbox{ for } a \in [a_0 + \eta, M],
\end{equation*}
and a.s. for $n$ large enough, $\sup_{\mathcal{D}} |\ell'_n(a) - \ell'_{\infty}(a)| < m_{\eta}/2$.
As a consequence,
\begin{equation*}
\ell'_{n}(a) > m_{\eta}/2 \mbox{ for } a \in [\varepsilon, a_0 - \eta] 
\end{equation*}
and
\begin{equation*}
\ell'_{n}(a) < - m_{\eta}/2 \mbox{ for } a \in [a_0 + \eta, M].
\end{equation*}
This implies that $\widehat{a}^{BO}_n \in [a_0 - \eta, a_0 + \eta]$ a.s. 
Since $\eta$ can be taken arbitrarily small, we prove the consistency of $\widehat{a}^{BO}_n$.
\end{proof}

We proceed to proving the asymptotic normality of $\widehat{a}^{BO}_n$.
To this end, let $v^{(k)}$ be the node being attached at time $k$.
The scaled log-likelihood can be written as
\begin{equation*}
\ell_n(a) = \frac{1}{n}\sum_{k= 1}^n F_k(a), 
\end{equation*}
where $F_k(a): = \log(a+ d_k(v^{(k)}) - 1) - \log(a+1-k^{-1})$
with the convention that $d_k(v^{(k)}) = 1$ if a self-loop is formed at epoch $k$.
Therefore,
\begin{equation*}
\ell_n'(a) = \frac{1}{n} \sum_{k = 1}^n f_k(a),
\end{equation*}
where
\begin{equation}
\label{eq:fka}
f_k(a) : = \frac{1}{a + d_k(v^{(k)}) - 1} - \frac{1}{a+1-k^{-1}}.
\end{equation}
Note that $\ell'_n(\widehat{a}^{BO}_n) = 0$.
By Taylor expanding $f_k$'s, we get
\begin{equation*}
0 = \sum_{k=1}^n f_k(\widehat{a}^{BO}_n) = \sum_{k=1}^n f_k(a_0) + (\widehat{a}_n^{BO} - a_0) \sum_{k = 1}^n f_k'(a^{\star}),
\end{equation*}
where $a^{\star} \in (a_0, \widehat{a}^{BO}_n)$. 
Consequently,
\begin{equation}
\label{eq:diffa}
\sqrt{n} (\widehat{a}^{BO}_n - a_0) 
= \left(-\frac{1}{n} \sum_{k=1}^n f_k'(a^{\star}) \right)^{-1} \left(\frac{1}{\sqrt{n}} \sum_{k=1}^n f_k(a_0) \right).
\end{equation}
The proof of Theorem \ref{thm:asympnormal} boils down to the following two lemmas.

\begin{lemma}
\label{prop:fk}
As $n \rightarrow \infty$,
\begin{equation}
\label{eq:convfk}
\frac{1}{\sqrt{n}} \sum_{k=1}^n f_k(a_0) \stackrel{(d)}{\longrightarrow} \mathcal{N}(0, \sigma^2),
\end{equation}
where $\sigma^2$ is defined by \eqref{eq:sigma}.
\end{lemma}

\begin{lemma}
\label{prop:fkprime}
As $n \rightarrow \infty$,
\begin{equation}
\frac{1}{n} \sum_{k = 1}^n f_k'(a^{\star}) \longrightarrow - \beta \quad \mbox{in probability},
\end{equation}
where $\beta$ is defined by \eqref{eq:beta}.
\end{lemma}

Now we consider the BPAM. 
Note that $\widehat{\pmb \pi}$ is the MLE of the Multinomial$(1, {\pmb \pi})$ distribution. 
It follows from standard exponential family theory \cite{Brown86} that $\widehat{\pmb \pi}$ is consistent and asymptotically normal.
The main difficulty of Theorem \ref{thm:consistencyBPA} is to prove that $\widehat{\pmb \gamma} \rightarrow {\pmb \gamma}^0$.
The idea is similar to that of Theorem \ref{thm:consistency}.
We prove that $M^n_{ij}/n  \rightarrow \theta^0_{ij}$ and $N_j^{n-1}/n \rightarrow p_j^0$ almost surely, with $\sum_{i,j = 1}^K \theta^0_{ij} = 1$ and $\sum_{j = 1}^K p^0_j = 2$.
\begin{proposition}
\label{prop:dadada}
For $1 \le i, j \le K$, $N^{n-1}_j/n \rightarrow p^0_j$ and $M^n_{ij}/n \rightarrow \theta^0_{ij}$ almost surely, where $(p_j^0; \, 1 \le j \le K)$ satisfies 
\begin{equation}
\label{eq:fixedp}
p^0_j \left(1- \sum_{l = 1}^K \frac{\pi_l^0 \gamma^0_{lj}}{\sum_{k=1}^K \gamma^0_{lk} p^0_k} \right) = \pi_j^0,
\end{equation}
and
\begin{equation}
\label{eq:fixedtheta}
\theta^0_{ij}: = \left\{ \begin{array}{ccl}
\frac{\pi_i^0 \gamma_{ij}^0 p_j^0}{\sum_{k=1}^K \gamma_{ik}^0 p_k^0} +  \frac{\pi_j^0 \gamma_{ji}^0 p_i^0}{\sum_{k=1}^K \gamma_{jk}^0 p_k^0} & \mbox{for} & i \ne j, \\[10 pt]
\frac{\pi^0 \gamma_{ii}^0 p_i^0}{\sum_{k=1}^K \gamma_{ik}^0 p_k^0} & \mbox{for} & i = j.
\end{array}\right.
\end{equation}
\end{proposition}
\begin{proof}
The idea is to establish a recursion for $N_j^n$, $1 \le j \le K$. 
Denote $(\mathcal{F}_n; \, n \ge 1)$ to be the natural filtration of the attachment process.
We have
\begin{multline}
\label{eq:Nrecursion}
\mathbb{E}(N_j^{n+1}|\mathcal{F}_n) - N^n_j = \sum_{l \ne j} \frac{\pi_l^0 \gamma^0_{lj}N_j^n }{\sum_k \gamma^0_{lk} N^n_k + \gamma^0_{ll}} \\
+ \pi^0_j \left[\sum_{l \ne j} \frac{\gamma^0_{lj} N^n_l}{\sum_k \gamma^0_{jk} N_k^n + \gamma^0_{jj}} + \frac{2 \gamma^0_{jj}(1+N_j^n)}{\sum_k \gamma^0_{jk} N_k^n + \gamma^0_{jj}}  \right],
\end{multline}
where the first term on the r.h.s. includes the contributions from a new node not in $C_j$ attached to an existing one in $C_j$, 
the second term from a new node in $C_j$ attached to an existing one not in $C_j$, and the third term from a new node in $C_j$ attached to an existing one in $C_j$.
Rearranging the terms in \eqref{eq:Nrecursion} yields
\begin{multline*}
\mathbb{E}(N_j^{n+1}|N_j^n) = \Bigg[1 + \sum_{l \ne j} \frac{\pi_l^0 \gamma^0_{lj}}{\sum_k \gamma^0_{lk} N^n_k + \gamma^0_{ll}} \\
+ \frac{2 \pi_j^0 \gamma^0_{jj}}{\sum_k \gamma^0_{jk} N_k^n + \gamma^0_{jj}}  \Bigg] N_j^n + \pi_j^0 \frac{\sum_{l \ne j}\gamma^0_{lj}N^n_l + 2 \gamma^0_{jj}}{\sum_k \gamma^0_{jk} N_k^n + \gamma^0_{jj}}.
\end{multline*}
Now we explain how \eqref{eq:fixedp} comes from. 
One expects $N^n_k \sim p_k^0 n$. Then the r.h.s. of the above expression is approximately
\begin{multline*}
\left[1 + \frac{1}{n} \sum_{l \ne j} \frac{\pi_l^0 \gamma^0_{lj}}{\sum_k \gamma^0_{lk}p_k^0} + \frac{1}{n} \frac{2 \pi_j^0 \gamma_{jj}^0}{\sum_{k} \gamma_{jk}^0 p_k^0} \right] N_j^n \\ + \pi_j^0 \frac{\sum_{l \ne j} \gamma^0_{lj} p^0_l}{\sum_k \gamma_{jk}^0 p_k^0} 
= \left[ 1 + \frac{1}{n} \sum_{l} \frac{\pi^0_l \gamma^0_{lj}}{\sum_k \gamma^0_{lk} p^0_k} \right] N_j^n + \pi_j^0.
\end{multline*}
Now by Lemma 4.1.2 in \cite{Durrett07}, $(p_j^0; \,1 \le j \le K)$ satisfies \eqref{eq:fixedp}. 
These facts can be justified by a routine argument as in Section 3, \cite{J13}.
The full details are left for the readers.

Now we move to the convergence of $M^n_{ij}/n$. Similarly, we get the recursion
\begin{multline}
\mathbb{E}(M_{ij}^{n+1}|\mathcal{F}_n) - M^n_{ij} \\
= \left\{ \begin{array}{ccl}
\frac{\pi_i^0 \gamma_{ij}^0 N_j^n}{\sum_k \gamma_{ik}^0 N^n_k + \gamma^0_{ii}} + \frac{\pi_j^0 \gamma_{ji}^0 N_i^n}{\sum_k \gamma_{jk}^0 N^n_k + \gamma^0_{jj}}& \mbox{for} & i \ne j, \\[10 pt]
\frac{\pi_i^0 \gamma_{ii}^0 (1+N_i^n)}{\sum_{k} \gamma_{ik}^0 N^n_k + \gamma^0_{ii}} & \mbox{for} & i = j.
\end{array}\right.
\end{multline}
It is easy to see from the asymptotics of $N^n_j$ that $M^n_{ij}/n$ converges to $\theta_{ij}^0$ satisfying \eqref{eq:fixedtheta}.
\end{proof}

The above result yields the limit of \eqref{eq:BPAgamma}:
\begin{equation}
\label{eq:BPAloginfty}
\ell_{\infty}^{BPA}({\pmb \gamma}):= 
\sum_{i,j = 1}^K \theta_{ij}^0 \log \gamma_{ij} - \sum_{i = 1}^K \pi_i^0 \log\left(\sum_{j=1}^K \gamma_{ij} p_j^0 \right).
\end{equation}
Note that $\ell^{BPA}_{\infty}$ is homogeneous of order $0$, i.e. $\ell^{BPA}_{\infty}(\theta {\pmb \gamma}) = \ell^{BPA}_{\infty}({\pmb \gamma})$ for each $\theta > 0$. 
It suffices to prove that $\ell_{\infty}^{BPA}$ attains its unique maximum at the equivalent class $[{\pmb \gamma}^0]$, 
from which the consistency of $\widehat{\pmb \gamma}$ follows. 
\begin{lemma}
\label{lem:bbb}
The function $\ell_{\infty}^{BPA}(\cdot)$ has a unique maximum at ${\pmb \gamma}_0$ up to a constant multiple.
\end{lemma}
All the remaining proofs are deferred to the Supplementary Material.

\section{Empirical results}
\label{s3}
\subsection{Buckley-Osthus model}

Theorems \ref{thm:consistency} and \ref{thm:asympnormal} show that the MLE of the Buckley-Osthus model is consistent and asymptotically normal. 
Here we present some simulation experiments to illustrate statistical inference for finite samples.
We pick three different values for the parameter, $a_0 \in \{0.5, 1.0, 2.0\}$.
For each parameter value, we generate $200$ realizations from the Buckley-Osthus dynamics \eqref{eq:BO}
with sample sizes $n = 100, 200, 500, 1000$.

We fit the Buckley-Osthus model by MLE, and calculate the MLE $\widehat{a}^{BO}$ via the gradient ascent algorithm.
So for each parameter value and sample size, there are $200$ estimates.
Figure \ref{fig:BObox} gives the boxplots of these MLEs for each parameter $a_0 \in \{0,5, 1.0, 2.0\}$ with sample sizes $n = 100. 200, 500, 1000$.
Table \ref{table:BO} summarizes a few statistics of the MLEs in different experiment settings.
\begin{figure}[ht]
\vskip 0.2in
\begin{center}
\centerline{\includegraphics[width=1 \columnwidth]{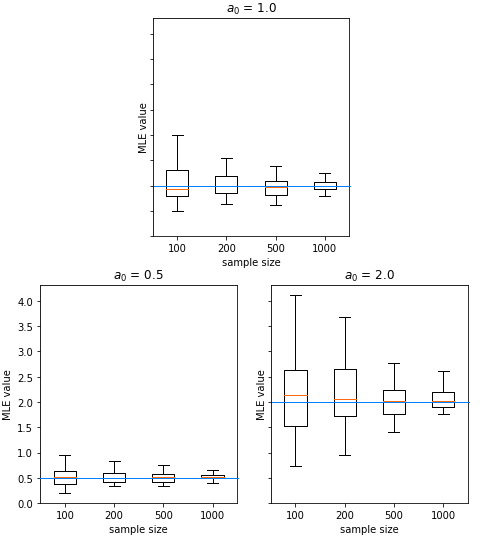}}
\caption{Boxplots of the MLEs for the Buckley-Osthus model with different parameter values and sample sizes.}
\label{fig:BObox}
\end{center}
\vskip -0.2in
\end{figure}

\begin{table}[ht]
\begin{center}
\begin{tabular}{ |l|c|c|c|c|} 
 \hline
$a_0 = 0.5$ & $n = 100 $  & $n = 200 $ & $n = 500 $ &  $n = 1000 $  \\ \hline
Mean & 0.543 & 0.511 & 0.510  & 0.506  \\ \hline
Median & 0.526 & 0.516 & 0.515  & 0.511  \\ \hline
Std & 0.200 & 0.114 & 0.100 & 0.062   \\ \hline
 \hline
$a_0 = 1.0$ & $n = 100 $  & $n = 200 $ & $n = 500 $ &  $n = 1000 $  \\ \hline
Mean & 1.125 & 1.098 & 0.977  & 1.011  \\ \hline
Median & 0.936 & 0.997 & 0.977  & 0.990  \\ \hline
Std & 0.645 & 0.365 & 0.182 & 0.123   \\ \hline
 \hline
$a_0 = 2.0$ & $n = 100 $  & $n = 200 $ & $n = 500 $ &  $n = 1000 $  \\ \hline
Mean & 2.540 & 2.354 & 2.081  & 2.085 \\ \hline
Median & 2.130 & 2.060 & 2.027  & 2.025  \\ \hline
Std & 1.708 & 1.008 & 0.487 & 0.251   \\ \hline
\end{tabular}
\end{center}
\caption{The mean, median and standard deviation of the MLEs for the Buckley-Osthus model with different parameter values and sample sizes.}
\label{table:BO}
\end{table}

It can be seen from Figure \ref{fig:BObox} and Table \ref{table:BO} that for $n \gtrapprox 500$, the mean/median of the MLEs are close to the true parameter values.
For each parameter value, the standard deviation of the MLEs decreases as the sample size increases.
Given the sample size, the standard deviation of the MLEs increases as the parameter value gets larger.
These observations agree with the theoretical findings in Section \ref{s23}.
\subsection{Block preferential attachment model}
Theorem \ref{thm:consistencyBPA} proves that the MLE of the BPAM is consistent. 
Here we corroborate this result with some simulations. 
We take $K = 2$ with $(\pi_1, \pi_2) = (0.3, 0.7)$, and $\gamma_{12} = \gamma_{21} = 0.5$ and $\gamma_{22} = 1.5$.
For this parameter setting, we generate $100$ realizations from the BPAM dynamics \eqref{eq:BPA} with sample sizes $n = 100, 200, 500, 1000$.

We calculate the MLEs $(\widehat{\pi}_1, \widehat{\pi}_2)$ by counting frequency, and the MLEs $(\widehat{\gamma}_{12}, \widehat{\gamma}_{22})$ via the gradient ascent algorithm.
For each sample size, we get $100$ sets of estimates. 
Figure \ref{fig:HPAMbox} provides boxplots of the MLEs with sample sizes $n = 100, 200, 500, 1000$.
Table \ref{table:HPAM} displays some statistics of the MLEs with different sample sizes.
\begin{figure}[ht]
\vskip 0.2in
\begin{center}
\centerline{\includegraphics[width=1 \columnwidth]{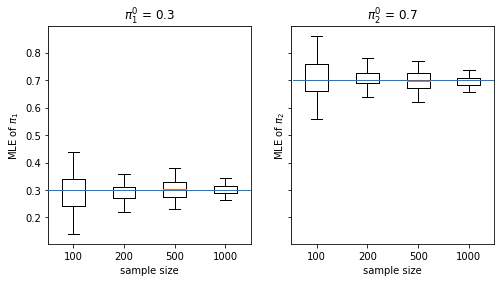}}
\centerline{\includegraphics[width=1 \columnwidth]{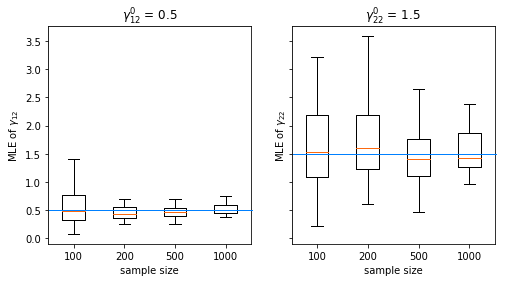}}
\caption{Boxplots of the MLEs of $(\pi_1, \pi_2, \gamma_{12}, \gamma_{22})$ for the HPAM with different sample sizes.}
\label{fig:HPAMbox}
\end{center}
\vskip -0.2in
\end{figure}

\begin{table}[ht]
\begin{center}
\begin{tabular}{ |l|c|c|c|c|} 
 \hline
$\pi_1 = 0.3$ & $n = 100 $  & $n = 200 $ & $n = 500 $ &  $n = 1000 $  \\ \hline
Mean & 0.291 & 0.297 & 0.302  & 0.303  \\ \hline
Median & 0.300 & 0.300 & 0.301  & 0.301  \\ \hline
Std & 0.071 & 0.036 & 0.033 & 0.024   \\ \hline
 \hline
$\pi_2 = 0.7$ & $n = 100 $  & $n = 200 $ & $n = 500 $ &  $n = 1000 $  \\ \hline
Mean & 0.709 & 0.703 & 0.698  & 0.697  \\ \hline
Median & 0.700 & 0.700 & 0.698  & 0.699  \\ \hline
Std & 0.071 & 0.036 & 0.033 & 0.024   \\ \hline
 \hline
$\gamma_{12} = 0.5$ & $n = 100 $  & $n = 200 $ & $n = 500 $ &  $n = 1000 $  \\ \hline
Mean & 0.736 & 0.486 & 0.476  & 0.520 \\ \hline
Median & 0.478 & 0.435 & 0.460  & 0.496  \\ \hline
Std & 0.837 & 0.203 & 0.122 & 0.098   \\ \hline
 \hline
$\gamma_{22} = 1.5$ & $n = 100 $  & $n = 200 $ & $n = 500 $ &  $n = 1000 $  \\ \hline
Mean & 1.680 & 1.555 & 1.480  & 1.529 \\ \hline
Median & 1.529 & 1.609 & 1.409  & 1.432  \\ \hline
Std & 1.842 & 1.457 & 0.588 & 0.514   \\ \hline
\end{tabular}
\end{center}
\caption{The mean, median and standard deviation of the MLEs for the HPAM with different sample sizes.}
\label{table:HPAM}
\end{table}

It can be observed from Figure \ref{fig:HPAMbox} and Table \ref{table:HPAM} that the MLEs $(\widehat{\pi}_1, \widehat{\pi}_2)$ get very close to the true values for $n \gtrapprox 100$, while the MLEs $(\widehat{\gamma}_{12}, \widehat{\gamma}_{22})$ seem to converge when $n \approx 1000$.
 This is partly due the stability of the gradient ascent algorithm to find the MLE of $\pmb{\gamma}$.
Also for each parameter, the standard deviation decreases as the sample size increases.
\subsection{Evolution of the Bitcoin network}
Bitcoin \cite{N08} is a distributed digital currency system which works without central governing.
Payments are processed by a peer-to-peer network of users connected through the internet. 
Here we apply the BPAM to the Bitcoin network. 
The data we use is available at \url{http://www.vo.elte.hu/bitcoin/zipdescription.htm} \cite{KPCV}.

In such a transaction, `A' is the receiver and `B' is the sender.
The data set consists of $30048983$ transactions up to December, 2013.
All these transactions are recorded with timestamps. 
In the Bitcoin network, there are two types of nodes: regular nodes and super nodes. 
Regular nodes represents normal users, while super nodes are professional miners or digital currency companies.
Super nodes are usually much more equipped than regular ones, and are more reliable in the transactions.
Hence, we model the Bitcoin network by the BPAM with $K = 2$.
The index `1' is used for super nodes, and `2' for regular nodes.
The idea is to fit the Buckley-Osthus model and the BPAM by MLE, and calculate the MLEs for the parameters of interest.

We preprocess the data by removing a few abnormal transactions, for instance, those related to the {\em SatoshiDice gambling} whose addresses start with `1Dice'.
Due to computation limits, we are unable to process all $3 \cdot 10^8$ transactions.
Instead, we consider the first $n = 5000, 10000, 20000, 50000$ transactions to get the corresponding parameter estimates.
We select a threshold of top $5 \%$ active users to distinguish super nodes from regular nodes.
Table \ref{table:Bitcoina} displays the MLE for the power-law exponent with different network sizes.
It can be seen that the exponent stabilizes around $4.4$ for $n \gtrapprox  20000$.
\begin{table}[h]
\begin{center}
\begin{tabular}{ |l|c|c|c|c|} 
 \hline
          & $n = 5000 $  & $n = 10000$ & $n = 20000$ &  $n =50000 $  \\ \hline
$\widehat{a}$ & 3.757 & 3.857  & 4.401  & 4.398  \\ \hline
\end{tabular}
\end{center}
\caption{The MLE for $a$ of the Bitcoin network.}
\label{table:Bitcoina}
\end{table}

Table \ref{table:Bitcoin} displays the MLEs for the BPAM parameters with different network sizes.
Observe that (1) the proportion of super nodes $\widehat{\pi}_1$ decreases as the network expands, (2) $\widehat{\gamma}_{12} > \widehat{\gamma}_{22}$ for $n = 5000$, and $\widehat{\gamma}_{12} < \widehat{\gamma}_{22}$ for $n \gtrapprox 10000$.
These phenomena can be interpreted as follows. 
At the early stage of the Bitcoin, many transactions were made through super nodes, and professional miners were the main player.
As the blockchain techniques are developed, more and more individuals participate in Bitcoin transactions.
This is the reason why we see an order change between $\widehat{\gamma}_{12}$ and $\widehat{\gamma}_{22}$.
The BPAM characterizes the evolution of the Bitcoin network.
\begin{table}[h]
\begin{center}
\begin{tabular}{ |l|c|c|c|c|} 
 \hline
          & $n = 5000 $  & $n = 10000$ & $n = 20000$ &  $n =50000 $  \\ \hline
$\widehat{\pi}_1$ & 0.410 & 0.312  & 0.225  & 0.199  \\ \hline
$\widehat{\pi}_2$ & 0.590 & 0.688 & 0.775 & 0.801  \\ \hline
$\widehat{\gamma}_{12} $ & 3.974 & 9.226 & 16.351 & 34.596   \\ \hline
$\widehat{\gamma}_{22} $ & 2.360 & 11.213 & 67.216 & 123.497   \\ \hline
\end{tabular}
\end{center}
\caption{The MLEs for $(\pi_1, \pi_2, \gamma_{12}, \gamma_{22})$ of the Bitcoin network.}
\label{table:Bitcoin}
\end{table}

\section{Conclusion and discussion}
\label{s4}
This paper deals with statistical estimation of two preferential attachment related models:
the Buckley-Osthus model, and the BPAM.
We prove the consistency of the MLE for both models, and corroborate the theory with simulation studies.
Though these models might be too simplistic for real-world networks (as shown in the Bitcoin example), they can capture some gross features and structural changes in the network.
Despite their limitations, these models may be used as a building piece for more flexible systems.

In the remaining of this section, we discuss a few open problems related to the BPAM.
As shown in Theorem \ref{thm:consistencyBPA}, if the memberships of all nodes are correctly classified, 
then the MLEs are consistent.
In fact, even if a small portion of memberships (e.g. of order $o(n)$) are misclassified, the consistency still holds.
Consider the case where all but the last node are correctly classified. 
It is easy to see that the statistics $(M_{ij}^n, N_j^n)$ is affected by $\pm 1$, and the limiting log-likelihood \eqref{eq:BPAloginfty} stays the same.
More generally, the degree of any node is of order $n^{\beta}$ for some $\beta < 1$. 
Thus, misclassifying any node will affect $(M_{ij}^n, N_j^n)$ by $\pm n^{\beta}$, which is absorbed by the $1/n$-scaling.
This leads to the problem of community detection of the BPAM.
\cite{HS18} gave a message passing algorithm to recover the memberships of $\Theta(1)$ nodes.
An important question is to which extent the memberships of nodes can be recovered for the BPAM.
\begin{op}
Is there any algorithm to recover the memberships of all nodes for the BPAM ?
If not, what is the maximum number of memberships which can be recovered ?
Is it of order $\Theta(n)$ ?
\end{op}

For the SBM, the MLEs of the marginal likelihood were shown to be asymptotically normal \cite{BCCZ13}. 
The key idea is that the likelihood with misclassifications is much smaller than that with all correct memberships.
However, this is not true for the BPAM since misclassifying a small portion of memberships will not affect much the likelihood. 
It is not clear whether the MLEs of the marginal likelihood \eqref{eq:marginalPA} give the correct memberships and good parameter estimates.
\begin{op}
Are the MLEs for the marginal likelihood \eqref{eq:marginalPA} consistent ?
\end{op}

From the computational viewpoint, the marginal likelihood \eqref{eq:marginalPA} is intractable so calculating the MLE is not efficient. 
In the case of the SBM, \cite{DPR08} applied the variational inference to remove the normalizing term so that the optimization can be easily solved by the EM algorithm. 
The corresponding estimates were also shown to be asymptotically normal \cite{BCCZ13}.
Here we can ask similar questions regarding the marginal likelihood for the BPAM.
\begin{op}
$~$
\begin{enumerate}
\item
Propose an approximation of the MLEs for the marginal likelihood \eqref{eq:marginalPA} which is computationally feasible.
\item
Is the approximation in (1) consistent ?
\item
How to choose the number of communities $K$ ? 
\end{enumerate}
\end{op}




\nocite{langley00}

\bibliography{example_paper}
\bibliographystyle{icml2020}

\newpage
\appendix

\section{Proof of Lemma \ref{prop:analysis}}
To simplify the notation, we omit the `BO' in the superscript.
We begin with a few algebraic identities for $p_k$. 
It is easy to see from \eqref{eq:N2}-\eqref{eq:N3} that 
\begin{equation}
\label{eq:recp}
p_k = \frac{k+a_0-2}{k+2a_0} p_{k-1} \quad \mbox{for } k \ge 2.
\end{equation}
Therefore, $\sum_{j = 2}^k (j+2a_0) p_j = \sum_{j = 2}^k (j+a_0-2) p_{j-1}$ which implies that 
\begin{equation}
\label{eq:tail}
p_{>k-1} = \frac{k+2a_0}{a_0+1} p_k \quad \mbox{for } k \ge 2.
\end{equation}
Further by summing both sides of \eqref{eq:tail}, we get $\sum_{k \ge 1} k p_k = 2$.
Observe that 
\begin{align*}
\ell'_{\infty}(a) & = \sum_{k \ge 0} \frac{p_{>k+1}}{a+k} - \frac{1}{a+1} \\
& = \sum_{k \ge 0} \frac{(k+2+2a_0) p_{k+2}}{(a_0 + 1)(a+k)} \\
& \qquad \qquad \qquad  - \frac{1}{a+1} \sum_{k \ge 0} \frac{k+2+2a_0}{k+a_0} p_{k+2} \\
& = \frac{a-a_0}{(a_0 + 1)(a+1)} \sum_{k \ge 0} \frac{(k+2+2a_0)(k-1)}{(k+a_0)(k+a)} p_{k+2} \\
& = \frac{a-a_0}{(a_0 + 1)(a+1)} \sum_{k \ge 0} \frac{k-1}{k+a} p_{k+1}.
\end{align*}
where the second equality is due to \eqref{eq:tail} and the last one stems from \eqref{eq:recp}.
In addition,
\begin{equation*}
\sum_{k \ge 0} \frac{k-1}{k+a} p_{k+1} \le \frac{1}{1+a} \sum_{k \ge 0} (k-1)p_{k+1} = 0,
\end{equation*}
where the last equality is due to the fact that $\sum_{k \ge 1} k p_k = 2$.
Therefore, $\ell'_{\infty}(\cdot)$ has a unique zero at $a_0$, and $\ell'_{\infty}(a) < 0$ if $a> a_0$ and $\ell'_{\infty}(a) > 0$ if $a < a_0$.
These imply that $\ell_{\infty}(\cdot)$ has a unique maximum at $a_0$.

Now we prove \eqref{eq:ellest}. We have
\begin{multline}
\label{eq:ellminus}
\ell'_n(a) - \ell'_{\infty}(a) = 
\sum_{k \ge 0} \frac{Z^n_{>k+1}/n - p_{>k+1}}{a+k} \\
+ \left(\frac{1}{n} \sum_{k =1}^n \frac{1}{a +1-k^{-1} } - \frac{1}{a+1}\right).
\end{multline}
Standard analysis shows that the second term on the r.h.s. of \eqref{eq:ellminus} goes to $0$ as $n \rightarrow \infty$.
Note that $(k+2) Z^n_{>k+1} = \sum_{j \ge k+2} (k+2) Z^n_j \le \sum_{j \ge k+2} j Z^n_j \le 2n$, which implies $Z^n_{>k+1}/n \le \frac{2}{k+2}$. 
Consequently,
\begin{multline}
\label{eq:supest}
\sup_{a > \varepsilon} \left|\sum_{k \ge 0} \frac{Z^n_{>k+1}/n - p_{>k+1}}{a+k}\right|
\le \sum_{k=0}^K \frac{|Z^n_{>k+1}/n - p_{>k+1}|}{\varepsilon + k}  \\ + \sum_{k > K} \frac{2}{(2+k)(a+k)} + \sum_{k > K} \frac{p_{>k+1}}{a+k}.
\end{multline}
The first term on the r.h.s. of \eqref{eq:supest} converges to $0$ a.s. by Theorem \ref{thm:convp}, and the last two terms can be made arbitrarily small for $K$ sufficiently large. 
Combining the above estimates yields the desired result.

\section{Proof of Lemma \ref{prop:fk}}

It follows easily from the definition that $(\sum_{k=1}^n f_k(a_0); \, n\ge 1)$ is a martingale.
To prove the convergence \eqref{eq:convfk}, it suffices to use Theorem 3.2 in \cite{HH80}
with the following conditions:
\begin{itemize}
\item
$n^{-1/2} \max_k |f_k(a_0)| \rightarrow 0$ in probability.
\item
$\mathbb{E}(n^{-1} \max_k f_k^2(a_0))$ is bounded in $n$.
\item
$n^{-1} \sum_{k=1}^n f_k^2(a_0) \rightarrow \sigma^2$ in probability.
\end{itemize}
The first two conditions are straightforward since $|f_k(a)| \le 2/a$. 
Now we check the last condition. 
Write
\begin{align*}
\frac{1}{n} \sum_{k=1}^n f_k^2(a_0) &= \frac{1}{n} \sum_{k=1}^n \frac{1}{(a_0 + d_k(v^{(k)})- 1)^2} + \frac{1}{n} \sum_{k=1}^n \frac{1}{(a_0+1-k^{-1})^2} \\
&  \quad  - \frac{2}{n}\sum_{k = 1}^ n \frac{1}{(a_0 + d_k(v^{(k)})- 1)(a_0+1-k^{-1})} \\
& := S_1 + S_2 - 2 S_3.
\end{align*}
Note that
\begin{equation*}
S_1 = \sum_{k \ge 0} \frac{Z_{>k+1}^n / n}{(a_0+k)^2} \longrightarrow \sum_{k \ge 0} \frac{p_{>k+1}}{(a_0+k)^2} \quad a.s.
\end{equation*}
which follows from Theorem \ref{thm:convp}.
By standard analysis, $S_2 \longrightarrow \frac{1}{(a_0+1)^2}$.
We decompose $S_3$ into two terms:
\begin{multline*}
S_3 = \frac{1}{n} \sum_{k=1}^n \frac{1}{a_0 + d_k(v^{(k)})- 1} \left(\frac{1}{a_0+1-k^{-1}} - \frac{1}{a_0+1}\right) \\ 
+ \frac{1}{(a_0+1)n}  \sum_{k=1}^n \frac{1}{a_0 + d_k(v^{(k)})- 1},
\end{multline*}
where the first term on the r.h.s. is bounded by $\frac{1}{an} \sum_{k = 1}^n \left(\frac{1}{a_0+1-k^{-1}} - \frac{1}{a_0+1}\right) \longrightarrow 0$, and the second term converges almost surely to $\frac{1}{a_0+1} \sum_{k \ge 0} \frac{p_{>k+1}}{a_0+k}$.
Combining all the above estimates yields the desired result.

\section{Proof of Lemma \ref{prop:fkprime}}

Write 
\begin{align*}
\frac{1}{n} \sum_{k = 1}^n f_k'(a^{\star}) &= \frac{1}{n} \sum_{k = 1}^n f_k'(a_0) + \frac{1}{n} \sum_{k=1}^n \left( f_k'(a^{\star}) - f_k'(a_0)\right) \\
& : = T_1 + T_2.
\end{align*}
Observe that $f_k'(a) = -f_k^2(a) - 2 f_k(a) \frac{1}{a+1-k^{-1}}$.
We get
\begin{equation}
\label{eq:T1}
T_1 = -\frac{1}{n} \sum_{k=1}^n f_k^2(a_0) - \frac{2}{n} \sum_{k=1}^n f_k(a_0) \frac{1}{a_0+1-k^{-1}}.
\end{equation}
The first term on the r.h.s. of \eqref{eq:T1} converges to $-\sigma^2$ as proved in Proposition \ref{prop:fk}.
Recall the definition of $S_2$, $S_3$. It is easy to see that
\begin{multline*}
\frac{1}{n}\sum_{k=1}^n f_k(a_0) \frac{1}{a_0+1-k^{-1}} = S_3 - S_2 \\
\longrightarrow \frac{1}{a+1} \sum_{k \ge 0} \frac{p_{>k+1}}{a_0+k} - \frac{1}{(a_0+1)^2}.
\end{multline*}
Therefore, $T_1 \longrightarrow -\beta$ in probability.
By standard analysis, $|T_2| \le C|a^{\star} - a_0|$ for some $C>0$. 
Note that $a^{\star} \in (a_0, \widehat{a}^{BO}_n)$. 
By Theorem \ref{thm:consistency}, $|a^{\star} - a_0| \longrightarrow 0$ which implies $T_2 \longrightarrow 0$.
The above estimates lead to the desired result.

\section{Proof of Lemma \ref{lem:bbb}}

As discussed in Section \ref{s23}, the consistency of $\widehat{\pmb \pi}$ follows from standard exponential family theory.
It suffices to prove that $\widehat{\pmb \gamma} \rightarrow {\pmb \gamma}^0$ almost surely.

Let us go back to the limit log-likelihood \eqref{eq:BPAloginfty}.
Observe that $\ell^{BPA}_{\infty}$ is homogeneous of order $1$, i.e. $\ell^{BPA}_{\infty}(a {\pmb \gamma}) = \ell^{BPA}_{\infty}({\pmb \gamma})$ for each $a > 0$. 
By taking the partial derivatives of \eqref{eq:BPAloginfty} and equating to $0$, we get
\begin{multline}
\frac{\partial}{\partial \gamma_{ij}}\ell^{BPA}_{\infty}({\pmb \gamma}) \\= \left\{ \begin{array}{ccl}
\frac{\theta_{ij}^0}{\gamma_{ij}} - \frac{\pi_i^0 p_j^0}{\sum_{k=1}^K \gamma_{ik} p_k^0} -  \frac{\pi_j^0 p_i^0}{\sum_{k=1}^K \gamma_{jk} p_k^0} & \mbox{for} & i \ne j,  \\[10 pt]
\frac{\theta_{ii}^0}{\gamma_{ii}} - \frac{\pi_i^0 p_i^0}{\sum_{k = 1}^K \gamma_{ik} p_k^0} & \mbox{for} & i = j. 
\end{array}\right.
\end{multline}

By \eqref{eq:fixedtheta}, we have $\nabla \ell_{\infty}^{BPA}({\pmb \gamma}_0) = {\pmb{0}}$, i.e. ${\pmb \gamma}_0$ is a stationary point of $\ell_{\infty}^{BPA}$. 
Now it suffices to prove Lemma \ref{lem:bbb} to conclude.

Note that $\ell^{BPA}_{\infty}({\pmb \gamma}) \rightarrow -\infty$ as ${\pmb \gamma} \in \partial \mathcal{D}$.
It suffices to prove that $\nabla \ell_{\infty}^{BPA}({\pmb \gamma}) = {\pmb{0}}$ has a unique solution. 
First $\partial \ell^{BPA}_{\infty}/\partial \gamma_{ii} = 0$ gives 
\begin{equation}
\label{eq:sumden}
\sum_{k = 1}^K \gamma_{ik} p_k^0 = \frac{\pi_i^0 p_i^0}{\theta_{ii}^0} \gamma_{ii}.
\end{equation}
By injecting \eqref{eq:sumden} into the equation $\partial \ell^{BPA}_{\infty}/\partial \gamma_{ij} = 0$, we get
\begin{equation}
\label{eq:ijii}
\frac{\theta_{ij}^0}{\gamma_{ij}} = \frac{\theta_{ii}^0 p_j^0}{p_i^0} \frac{1}{\gamma_{ii}} + \frac{\theta_{jj}^0 p_i^0}{p_j^0} \frac{1}{\gamma_{jj}}.
\end{equation}
Consequently, the values of $(\gamma_{ij}; \, i \ne j)$ is uniquely determined by those of $(\gamma_{ii}; \, 1 \le i \le K)$.
By injecting \eqref{eq:ijii} into \eqref{eq:sumden}, we get a system of equations on $(\gamma_{ii}; \, 1 \le i \le K)$:
\begin{equation}
\label{eq:ii}
\sum_{k = 1}^K \theta^0_{ik} \left( \frac{\theta_{ii}^0 p_j^0}{p_i^0} \frac{1}{\gamma_{ii}} + \frac{\theta_{kk}^0 p_i^0}{p_k^0} \frac{1}{\gamma_{kk}}\right)^{-1} p_k^0 = \frac{\pi_i^0 p_i^0}{\theta_{ii}^0} \gamma_{ii}
\end{equation}
For $K = 2$, it is easy to solve the equations together with the constraints $\gamma_{11}= 1$. 
For $K \ge 3$, the explicit solution is not available but we prove that the equations have a unique solution.
To illustrate, we consider the generic case $K =3$. All other cases can be proceeded in a similar way.

Let $x_1: = \frac{\theta^0_{11} p_2^0}{p_1^0} \gamma_{22} \left( \frac{\theta^0_{11} p_2^0}{p_1^0} \gamma_{22} + \frac{\theta^0_{22} p_1^0}{p_2^0} \gamma_{11}\right)^{-1}$, 
$x_2: = \frac{\theta^0_{11} p_3^0}{p_1^0} \gamma_{33} \left( \frac{\theta^0_{11} p_3^0}{p_1^0}\gamma_{33} + \frac{\theta^0_{33} p_1^0}{p_3^0} \gamma_{11} \right)^{-1}$, 
and $x_3: = \frac{\theta^0_{33} p_2^0}{p_3^0} \gamma_{22} \left( \frac{\theta^0_{33} p_2^0}{p_3^0} \gamma_{22}+ \frac{\theta^0_{22} p_3^0}{p_2^0} \gamma_{33} \right)^{-1}$.
The equations \eqref{eq:ii} give
\begin{equation}
\label{eq:system}
\left\{\begin{array}{lcl}
\theta_{12}^0 x_1 + \theta_{13}^0 x_2 = \pi_{1}^0 - \theta_{11}^0,  \\[10 pt]
\theta_{21}^0 (1-x_1) + \theta_{23}^0 (1 - x_3) = \pi_2^0 - \theta_{22}^0, \\ [10 pt]
\theta_{31}^0(1-x_2) + \theta_{32}^0 x_3 =  \pi_3^0 - \theta_{33}^0.
\end{array}\right.
\end{equation}
It suffices to prove that the equations \eqref{eq:system} have a unique solution.
Observe that the system \eqref{eq:system} has a solution $(x_1^0, x_2^0, x_3^0)$ by taking $\gamma_{ii} = \gamma_{ii}^0$.
Algebraic manipulation shows that the set of solutions to \eqref{eq:system} has dimension $1$, with form
\begin{equation*}
(x_1, x_2, x_3) = (x_1^0, x_2^0, x_3^0) + \lambda (1, -\theta_{12}^0/\theta_{13}^0, -\theta_{21}^0/\theta_{23}^0).
\end{equation*}
Consequently,
\begin{equation*}
\frac{\gamma_{11}}{\gamma_{22}} = \frac{\theta_{11}^0 (p_2^0)^2}{\theta_{22}^0 (p_1^0)^2} \frac{1-x_0-\lambda}{x_0 + \lambda}, \quad \frac{\gamma_{11}}{\gamma_{13}} = \frac{\theta_{11}^0 (p_3^0)^2}{\theta_{33}^0 (p_1^0)^2} \frac{1-y_0+ \lambda\theta_{12}^0 \theta_{13}^0 }{y_0 - \lambda\theta_{12}^0 \theta_{13}^0}
\end{equation*}
\begin{equation*}
\frac{\gamma_{33}}{\gamma_{22}} =  \frac{\theta_{33}^0 (p_2^0)^2}{\theta_{22}^0 (p_3^0)^2} \frac{1-z_0 + \lambda \theta_{21}^0/\theta_{23}^0 }{z_0 - \lambda \theta_{21}^0/\theta_{23}^0},
\end{equation*}
which implies that
\begin{equation}
\label{eq:cyclic}
\frac{1-x_0-\lambda}{x_0 + \lambda} = \frac{(1-y_0+ \lambda\theta_{12}^0 \theta_{13}^0 )(1-z_0 + \lambda \theta_{21}^0/\theta_{23}^0)}{(y_0 - \lambda\theta_{12}^0 \theta_{13}^0)(z_0 - \lambda \theta_{21}^0/\theta_{23}^0)}. 
\end{equation}
Note that the l.h.s. of \eqref{eq:cyclic} is decreasing in $\lambda$ while the r.h.s. is increasing in $\lambda$. 
Thus, $\lambda = 0$ is the only solution which proves the uniqueness.

\end{document}